\documentclass[a4paper,12pt]{article} 
\usepackage{amsmath,fancyhdr,amssymb,amsthm, geometry, enumerate, graphicx ,amsfonts,hyperref,mathrsfs}
\title{Specification properties for non-autonomous discrete systems}
\author{Mohammad Salman$^a$, \  Ruchi Das$^{a,\dag}$}
\title{Specification properties for non-autonomous discrete systems}
\theoremstyle{definition}
\newtheorem{defn}{Definition}[section]
\providecommand{\keywords}[1]{\textbf{Key words and phrases :} #1}
\providecommand{\msc}[1]{\textbf{Mathematics Subject Classification(2010)} #1}
\theoremstyle{plain}
\newtheorem{thm}{Theorem}[section]

\newtheorem{cor}[defn]{Corollary}

\theoremstyle{definition}
\newtheorem{exm}{Example}[section]
\newtheorem{rmk}{Remark}[section]
\begin{document}
\date{}
\maketitle

\begin{abstract}
 In this paper  notions of strong specification property and quasi-weak specification property for non-autonomous discrete systems are introduced and studied. It is shown that these properties are dynamical properties and are preserved under  finite product. It is proved that a $k$-periodic non-autonomous system on intervals having weak specification is Devaney chaotic whereas if the system has strong specification then the result is true in general. Specification properties of induced systems on hyperspaces and probability measures spaces are also studied. Examples/counter examples are provided wherever necessary to support results obtained. 
\end{abstract}


\msc{37B55; 54H20; 37B20; 54B20}
\bigskip\renewcommand{\thefootnote}{\fnsymbol{footnote}}
\footnotetext{\hspace*{-5mm}
\renewcommand{\arraystretch}{1}
\begin{tabular}{@{}r@{}p{11cm}@{}}
$^\dag$& the corresponding author. \emph{Email addresses}: rdasmsu@gmail.com (R. Das), salman25july@gmail.com (M. Salman)\\
$^a$&Department of Mathematics, University of Delhi, Delhi-110007, India
\end{tabular}}

\vspace{-2mm}
\section{Introduction}
Dynamical system  is a very well developed branch of mathematics. In its contemporary formulation,
the theory grows directly from advances in understanding complex and nonlinear systems in physics and mathematics.  Over the last 40 years with the discovery of chaos lots of research has been done in autonomous dynamical systems. The first paper that described chaos in a mathematically rigorous way is that of Li and Yorke \cite{11}. Since then the research on chaos has had a great influence on modern science. Specification property is an interesting and rather stronger notion of chaos. This property is closely related to the study of hyperbolic systems. Roughly speaking, by specification property we mean that any $k$ finite pieces of orbits can be approximated by some periodic orbit, provided the period is large enough. In 1971, Bowen  introduced the specification property on Axiom A diffeomorphisms \cite{3}. This property seems technical but is satisfied by many important maps like shift systems, Anosov diffeomorphisms, etc. For more details, one can refer to \cite{21}. The relevance of the specification property is that it plays a key role in the study of the uniqueness of equilibrium states, large deviation theory, multi-fractal analysis, etc. In \cite{22}, it is seen that the irregular set for maps with the specification property has full topological pressure.  Various types of specification properties including weak specification property, almost specification property \cite{5, 10, 101}, approximate product property \cite{14}, for autonomous discrete dynamical systems are intensively studied from an ergodic view point as well as algebraic view point. In \cite{1}, authors have studied specification on operators and proved that this property is equivalent to the notion of Devaney chaos for backward shift operators on Banach sequence spaces. Recently, specification property for uniform spaces has been defined and studied in \cite{18}.

Most of the real-world problems like weather and climate prediction, heartbeat patterns, spread of infectious diseases, etc., are time variant, that is, they involve time-dependent parameters, modulation and various other effects. Thus, non-autonomous systems are more flexible tools for the description and study of real world process. Dynamics of such systems are more complicated than autonomous dynamical systems and the variety of dynamical behaviour that can be represented is much richer. The concept of non-autonomous discrete dynamical systems was introduced by Kolyada and Snoha \cite{9}, in 1996. Since then non-autonomous systems are widely studied and have remarkable applications   \cite{4, 25}. Chaos for the non-autonomous dynamical systems was introduced by Tian and Chen \cite{23}. In 2014, authors have introduced the concept of weak specification property (WSP) for non-autonomous systems \cite{12}. They have related specification property to topological mixing, shadowing and distality in non-autonomous systems. In 2017, authors  have proved that any non-autonomous system having the weak specification property has positive topological entropy and that each uniformly expanding non-autonomous system satisfies the shadowing and the specification properties \cite{17}. Recently, in 2018, authors have studied stronger forms of sensitivity for non-autonomous discrete dynamical systems \cite{15, 24}.

The paper is organized as follows. In section 2, we give prerequisites for development of rest of the paper. In Section 3, we introduce the concept of strong (SSP) and quasi-weak specification (QSP) properties for non-autonomous systems. It is shown that if the $k$-peridic non-autonomous system has specification properties, then the corresponding autonomous system generated by the composition of $k$ members also has specification. Moreover, specification properties are dynamical properties and preserved under finite product. In Section 4, we show that for a non-autonomous system QSP is equivalent to topological mixing on a compact metric space. It is shown that a $k$-periodic non-autonomous system on intervals having WSP or QSP is Devaney chaotic and if the system has SSP, then this result is true in general. In Section 5, we  proved that if the non-autonomous system has SSP, then the corresponding system induced on the hyperspaces also has SSP and this result holds both ways for WSP as well as QSP. It is proved that if the non-autonomous system has SSP, then the corresponding system induced on the probability measures spaces also has SSP and the result is true both ways for QSP.

\section{Preliminaries}
 In the present paper, we consider the following non-autonomous discrete dynamical system:
\begin{equation}\label{eq} x_{n+1} = f_n(x_n) ,  \ \ \ n \geq 1, \end{equation}
where $(X,d)$ is a \emph{compact metric space} and $f_n: X \rightarrow X$ is a continuous map, for each $n\geq 1$. When $f_n = f$, for each $n\geq 1$, then system (\ref{eq}) becomes autonomous system. Denote $f_{1,\infty} := \{f_n\}_{n=1}^{\infty}$, and for all positive integers $i$ and $n$, $f_n^i := f_{n+i-1}\circ\cdots\circ f_n$, $f_n^0 := id$. For the system (\ref{eq}), the $orbit$ of any point $x\in X$  is the set, $\{f_1^n(x) : n\geq 0\} = \mathcal{O}_{f_{1,\infty}}(x)$. We say that $(X, f_{1,\infty})$ is a $k$-$periodic$ discrete system, if there exists $k\in\mathbb{N}$ such that $f_{n+k}(x) = f_n(x)$, for any $x\in X$ and $n\in\mathbb{N}$. Throughout this paper,  $f_{1,\infty}$ denotes the surjective family, that is, each $f_i$ is surjective. For a non-autonomous system $(X,f_{1,\infty})$, we put $X^{2}$ = $X \times X$ and $(f_{1,\infty})^{2}$ = ($g_1$, $g_2$, \ldots  , $g_n$, \ldots), where $g_{n}$ = $f_{n}$ $\times$ $f_{n}$, for each positive integer $n$. Therefore, $(X^{2},(f_{1,\infty})^{2})$ is a non-autonomous dynamical system. We have, 
$g_1^n = g_n \circ g_{n-1}\circ\cdots\circ g_2\circ g_1
               = (f_n \times f_n)\circ (f_{n-1} \times f_{n-1})\circ\cdots\circ (f_2 \times f_2) \circ (f_1 \times f_1) = f_{1}^{n}\times f_{1}^{n}$. Similarly we can define $(X^{m},(f_{1,\infty})^{m})$ in general for any positive integer $m$. Let $(X, d_1)$ and $(Y, d_2)$ be two metric spaces, then the product metric $\tilde{d}$ on $X\times Y$ is defined by $\tilde{d}((x_1, y_1), (x_2, y_2)) = \max\{d_1(x_1, x_2), d_2(y_1, y_2)\}$, for all $(x_1, y_1),(x_2, y_2)\in X\times Y.$
 Let $B_d(x, \epsilon)$ be the open ball of radius $\epsilon>0$ and  center $x$ and $\mathbb{N}$ be the set of natural numbers.

The symbol $\cal{K}$($X)$ denotes the hyperspace of all non-empty compact subsets of $X$ endowed with the \textit{Vietoris Topology}. A basis  for Vietoris topology is given by the sets, $\langle U_1, U_2, \ldots , U_k\rangle$ = $\{K \in$ $\cal{K}$($X)$: $K \subseteq \bigcup_{i=1}^{k} U_{i}$ and $K\cap U_{i}$  $\ne \emptyset$, for each $i$ $\in \{1, 2, \ldots ,k\}$\}, where $U_1, U_2, \ldots ,U_k$ are non-empty open subsets of $X$. Let $x \in X$, $A \in$ $\cal{K}$($X)$ and $N(A, \epsilon)$ = $\bigcup_{a \in A} B_d(\epsilon, a)$. The Hausdorff metric in $\cal{K}$($X)$  induced by $d$, denoted by $\cal{H}$ is defined  by $\cal{H}$($A, B) = \inf \{ \epsilon > 0: A \subseteq N(B, \epsilon) \  \text{and} \  B \subseteq N(A, \epsilon)\}$,  where $A$, $B \in$ $\cal{K}$($X)$. The topology induced by the Hausdorff metric on $\mathcal{K}(X)$ coincides with the Vietoris topology if and only if the space $X$ is compact \cite{8}. Under this topology, the set of all finite subsets of $X$, $\mathcal{F}(X)$, is dense in $\mathcal{K}(X)$. Let $(X, f_{1,\infty})$ be a non-autonomous  dynamical system and $\overline{f}_n$ the  continuous function on $\cal{K}$($X)$ induced by $f_n$, for each $n\in\mathbb{N}$. Then the sequence $\overline{f}_{1,\infty}$ = ($\overline{f}_1,\ldots, \overline{f}_n, \ldots )$ induces a non-autonomous discrete dynamical system ($\cal{K}$($X), \overline{f}_{1,\infty})$, where $\overline{f}_1^n = \overline{f}_n \circ \cdots \circ \overline{f}_2\circ \overline{f}_1$.  Clearly, $\overline{f}_1^n = \overline{f^n_1}$.

Let $\mathcal{B}(X)$ be the $\sigma$-algebra of Borel subsets of $X$ and $\mathcal{M}(X)$ be the set of all \emph{Borel probability measures} on $(X, \mathcal{B}(X))$ and $\mathcal{M}(X)$ be equipped with the \emph{Prohorov metric $\mathcal{D}$} defined by $\mathcal{D}(\mu, \nu)$ = $\inf\{\epsilon: \mu(A)\leq \nu(N(A, \epsilon))+\epsilon \ \text{and} \ \nu(A)\leq \mu(N(A, \epsilon))+\epsilon$, for each $A\in\mathcal{B}(X)\}$. It is known that topology induced by $\mathcal{D}$ is weak*-topology \cite{7}. For $x\in X$, $\delta_x\in \mathcal{M}(X)$ denotes \emph{Dirac point measure}, given by $\delta_x(A) = 1$, if $x\in A$ and $0$ otherwise. Let $\mathcal{M}_n(X) =\{(\sum_{i=1}^n\delta_{x_i})/n : x_i\in X \ \text{(not necessarily distinct)}\}$ and $\mathcal{M}_{\infty}(X) = \bigcup_{n\in\mathbb{N}}\mathcal{M}_n(X)$. It is known that $\mathcal{M}_{\infty}(X)$ is dense in $\mathcal{M}(X)$ and each $\mathcal{M}_n(X)$ is closed in $\mathcal{M}(X)$ \cite{2}. For a non-autonomous system $(X, f_{1,\infty})$, we consider a non-autonomous induced system $(\mathcal{M}(X), \widetilde{f}_{1,\infty})$, where each $\widetilde{f}_{i}: \mathcal{M}(X)\to \mathcal{M}(X)$ is induced continuous function and $\widetilde{f}_{1}^n(\mu)(A) = \mu(f_1^{-n}(A))$, $\mu\in \mathcal{M}(X)$, $A\in\mathcal{B}(X)$ and $f_1^{-n} = (f_1^n)^{-1}$.

\begin{defn}\cite{23} A non-autonomous system $(X, f_{1,\infty})$ is said to be \emph{topologically transitive}, if for each pair of non-empty open subsets $U$, $V$ of $X$, there exists $n\in\mathbb{N}$ such that $f_1^n(U)\cap V\ne\emptyset$. For any two non-empty  open subsets $U$ and $V$ of $X$ denote, 
$N_{f_{1,\infty}}(U,V)= \{n\in \mathbb{N} : f_1^n(U)\cap V\ne \emptyset \}$. Therefore, $(X, f_{1,\infty})$ is transitive if $N_{f_{1,\infty}}(U,V)\ne\emptyset$, for any pair of non-empty open subsets $U,$ $V$ of $X$.
\end{defn}
\begin{defn}\cite{13}  A point $x\in X$ is said to be \emph{periodic}, for the non-autonomous system $(X, f_{1,\infty})$, if there exists $n\in\mathbb{N}$ such that $f_1^{nk}(x) = x$, for every $k\in\mathbb{N}$.
\end{defn}
\begin{defn}\cite{23} The system $(X, f_{1,\infty})$ is said to exhibit \emph{sensitive dependence on initial conditions} if there exists $\delta>0$ such that, for every $x\in X$ and any neighborhood $U$ of $x$, there exist $y\in U$ and $n\in\mathbb{N}$ with $d(f_1^n(x), f_1^n(y))>\delta$; $\delta>0$ is called a constant of sensitivity.\end{defn}
\begin{defn} A non-autonomous system $(X, f_{1,\infty})$ is said to be \emph{chaotic in the sense of Devaney} on $X$ if   \begin{enumerate}
\item
It is topologically transitive on $X$;
\item
 It has a dense set of periodic points;
\item
It has sensitive dependence on initial conditions on $X.$
\end{enumerate}
\end{defn}
A non-autonomous system $(X, f_{1,\infty})$ is \emph{Wiggins chaotic}, if it satisfies conditions $(1)$ and $(3)$ only in the above definition.
\begin{defn} A non-autonomous system $(X, f_{1,\infty})$ is said to be \emph{topologically mixing}, if there exists $n\in\mathbb{N}$ such that $N_{f_{1,\infty}}(U,V) \supseteq [n, \infty)$, for any pair of non-empty open subsets $U,$ $V$ of $X$.
\end{defn}
\begin{defn}\cite{12} A non-autonomous system $(X, f_{1,\infty})$ is said to have \emph{weak specification property (WSP)}, if for every $\epsilon>0$, there exists $N$ such that for every choice of points $x_1$, $x_2$, \ldots, $x_s\in X$ and any sequence $a_1\leq b_1< a_2 \leq b_2< \cdots <a_s \leq b_s$ of non-negative integers with $a_j - b_{j-1}> N$, ($2\leq j\leq s$), there is a point  $z\in X$  satisfying
\[d(f_1^j(z), f_1^j(x_i))<\epsilon, \ \text{for all} \ a_i\leq j\leq b_i, \ 1\leq i\leq s.\]
\end{defn}
\begin{defn}
Let $(X, f_{1,\infty})$ and $(Y,g_{1,\infty})$ be two non-autonomous discrete dynamical systems. Let $h: X \to Y$ be  such that $ g_n(h(x))=h(f_n(x)), \ \text{for each}$  $n \in \mathbb{N} \  \text{and each}  \  x \in X.$
If $h$ is continuous,  surjective map (homeomorphism), then $f_{1,\infty}$ and $g_{1,\infty}$ are said to be \emph{topologically semi-conjugate (topologically conjugate)}.
\end{defn}
\section{Strong and Quasi-weak Specification Properties for Non-Autonomous Systems}
In this section, we first introduce  concepts of strong and quasi-weak specification properties for non-autonomous systems. It is shown that if a $k$-periodic non-autonomous system has specification properties, then the corresponding autonomous system generated by the composition of $k$ members, also has specification properties. It is proved that specification properties are dynamical properties and are preserved under finite product.
\begin{defn}A non-autonomous system $(X, f_{1,\infty})$ is said to have \emph{strong specification property (SSP)}, if for every $\epsilon>0$, there exists $M(\epsilon)$ such that for every choice of points $x_1$, $x_2$, \ldots, $x_k\in X$ and any sequence $a_1\leq b_1< a_2 \leq b_2< \cdots <a_k \leq b_k$ of non-negative integers with $a_j - b_{j-1}> M(\epsilon)$, ($2\leq j\leq k$) and any $p> M(\epsilon) + b_k - a_1$, there exists a periodic point $z\in X$ with period $p$ satisfying
\[d(f_1^j(z), f_1^j(x_i))<\epsilon, \ \text{for all} \ a_i\leq j\leq b_i, \ 1\leq i\leq k.\] 
 We shall  call the above defined  specification property for $k =2$ as the  \emph{periodic specification property (PSP)}. Now, we define a weaker form of specification property.\end{defn}
\begin{defn}
A non-autonomous system $(X, f_{1,\infty})$ is said to have the \emph{quasi-weak specification property (QSP)}, if for any $\epsilon>0$, there exists a positive integer $M(\epsilon)$ such that for any $x_1$, $x_2\in X$ and any $n\geq M(\epsilon)$, there is a point $z\in X$ such that $d(z, x_1)<\epsilon$ and $d(f_1^n(z), f_1^n(x_2))<\epsilon$.

Clearly, we have \ SSP $\implies$ WSP $\implies$ QSP.
\end{defn}
\begin{rmk}In \cite{12}, authors have proved that the non-autonomous system having WSP is topologically mixing but not conversely. Since SSP implies WSP, therefore if the non-autonomous system $(X, f_{1,\infty})$ has SSP, then it is topologically mixing but converse is not true.
\end{rmk}
\begin{thm}\label{K} Let $(X, f_{1,\infty})$ be a $k$-periodic non-autonomous system and $g = f_k\circ f_{k-1}\circ\cdots \circ f_1$. If $(X, f_{1,\infty})$ has SSP, then the corresponding autonomous system $(X, g)$ also has SSP.
\end{thm}
\begin{proof}Let $\epsilon>0$ be arbitrary and $M(\epsilon)$ be the positive integer corresponding to $\epsilon$ as in the definition of SSP. Let $M_1 = [M(\epsilon)/k] + 1$, where $[ \ ]$ denotes the greatest integer function. Consider a sequence $a_1\leq b_1< a_2 \leq b_2< \cdots <a_l \leq b_l$ of non-negative integers with $a_j - b_{j-1}> M_1$ ($2\leq j\leq l$) and any $p> M_1 + b_l - a_1$, then  $ka_1\leq kb_1< ka_2 \leq kb_2< \cdots <ka_l \leq kb_l$ and $ka_j - kb_{j-1}>kM_1 > M(\epsilon)$ and $pk> M(\epsilon) + kb_l - ka_1$. Since $(X, f_{1,\infty})$ has SSP, therefore there exists a periodic point $z\in X$ with period $kp$ such that $d(f_1^i(z), f_1^i(x_m))<\epsilon$, for all $ka_m\leq i \leq kb_m$, $1\leq m\leq l$. Taking $i = jk$ , we get $a_m\leq j\leq b_m$, $1\leq m\leq l$ and $d(f_1^{jk}(z), f_1^{jk}(x_m))<\epsilon$. Now, the system $(X, f_{1,\infty})$ is $k$-periodic, so $f_1^{jk} = (f_1^k)^j$. Therefore, $d((f_1^k)^j(z), (f_1^k)^j(x_m))<\epsilon$, that is, $d(g^j(z), g^j(x_m))<\epsilon$, for all $a_m\leq j\leq b_m$, $1\leq m\leq l$. Also, $f_1^{kps}(z) = z$, for each $s\in\mathbb{N}$ and in particular for $s = 1$, we get that $(f_1^k)^p(z) = z$ and hence $g^p(z) = z$, for $p> M_1 + b_l - a_1$. Thus, the autonomous system $(X, g)$ has SSP.
\end{proof}

\begin{thm} Let $(X, d_1)$, $(Y, d_2)$ be two metric spaces and $(X, f_{1,\infty})$, $(Y, g_{1,\infty})$ be the two non-autonomous systems such that $f_{1,\infty}$ is topologically semi-conjugate to $g_{1,\infty}$. If $f_{1,\infty}$ has SSP, then $g_{1,\infty}$ also has SSP.
\end{thm}
\begin{proof}Since $f_{1,\infty}$ is topologically semi-conjugate to $g_{1,\infty}$, therefore there exists a continuous surjective map $h : X\to Y$ such that $h\circ f_n =  g_n\circ h$, for each $n\in\mathbb{N}$. Let $\epsilon>0$ be arbitrary, then by uniform continuity of $h$ for given $\epsilon>0$, there exists $\delta>0$, such that \begin{equation}\label{F1}d_1(x, y)<\delta  \ \implies d_2(h(x), h(y))<\epsilon, \  \text{for any} \ x, \ y\in X. \end{equation} Let and $M(\delta)$ be as in the definition of SSP for $f_{1,\infty}$ and for the given $\epsilon$, take $M(\epsilon) = M(\delta)$. Consider a sequence $a_1\leq b_1< a_2 \leq b_2< \cdots <a_k \leq b_k$ of non-negative integers with $a_j - b_{j-1}> M(\delta)$ ($2\leq j\leq k$) and $y_1$, $y_2$, \ldots, $y_k\in Y$. Now as $h$ onto, so corresponding to each $y_i$, there exist $x_i\in X$ such that $h(x_i) = y_i$, for each $i = 1, 2, \ldots, k$. Since $f_{1,\infty}$ has SSP, therefore there exists $x\in X$ such that $d_1(f_1^j(x), f_1^j(x_i))<\delta$, for $a_i\leq j\leq b_i$ and $f_1^{pm}(x) = x$, for $p>M(\epsilon)+b_k-a_1$ and each $m\in\mathbb{N}$. Therefore, by (\ref{F1}), we have
\begin{equation}\label{F2} 
d_2(h(f_1^j(x)), h(f_1^j(x_i)))<\epsilon, \ \text{for} \ a_i\leq j\leq b_i \end{equation}  Taking $y = h(x)$ and using $h\circ f_n =  g_n\circ h$, we get that for any $j\in\mathbb{N}$, $g_1^j(y) = g_j\circ \cdots \circ g_1(h(x)) = g_j\circ \cdots \circ h(f_1(x)) = \cdots = h(f_1^j(x))$. Thus, using (\ref{F2}), we get $d_2(g_1^j(y), g_1^j(y_i)) = d_2(h(f_1^j(x)), h(f_1^j(x_i)))<\epsilon$, for $a_i\leq j\leq b_i$, $1\leq i\leq k$. Also $g_1^{pm}(y) = h(f_1^{pm}(x)) = h(x) = y$, for each $m\in\mathbb{N}$. Therefore, $(Y, g_{1,\infty})$ has SSP.
\end{proof}
\begin{cor}\label{ck} Let $(X, f_{1,\infty})$ and $(Y, g_{1,\infty})$ be two non-autonomous systems such that $f_{1,\infty}$ is topologically conjugate to $g_{1,\infty}$. Then $f_{1,\infty}$ has SSP if and only if $g_{1,\infty}$ has SSP.
\end{cor}
\begin{thm}\label{KK} Let $(X, d_1)$ and $(Y, d_2)$ be  two metric spaces. Then  non-autonomous systems $(X, f_{1,\infty})$ and $(Y, g_{1,\infty})$ have SSP if and only if the non-autonomous systems $(X\times Y, f_{1,\infty}\times g_{1,\infty})$ has SSP.
\end{thm}
\begin{proof}Let $(X, f_{1,\infty})$ and $(Y, g_{1,\infty})$ have SSP and $h_{1,\infty} = f_{1,\infty}\times g_{1,\infty}$. We show that $h_{1,\infty}$ has SSP. Let $\epsilon>0$ be arbitrary and $M_1(\epsilon)$, $M_2(\epsilon)$ be as in the definition of SSP for $f_{1,\infty}$ and $g_{1,\infty}$, respectively and $M(\epsilon) = \max\{M_1(\epsilon), M_2(\epsilon)\}$. Consider a sequence $a_1\leq b_1< a_2 \leq b_2< \cdots <a_k \leq b_k$ of non-negative integers with $a_j - b_{j-1}> M(\epsilon)$ ($2\leq j\leq k$) and $(x_i, y_i)\in X\times Y$, for $i = 1, 2, \ldots, k$. Now as $M(\epsilon)\geq M_1(\epsilon)$ and $M(\epsilon)\geq M_2(\epsilon)$, so using SSP of $f_{1,\infty}$ and $g_{1,\infty}$, there exist $x\in X$ and $y\in Y$ such that 
\begin{equation}\label{p1}
d_1(f_1^j(x), f_1^j(x_i)) <\epsilon, \ f_1^{mp_1}(x) = x, \ \text{for} \ p_1>M(\epsilon)+b_k-a_1,
\end{equation}
\begin{equation}\label{p2}
d_2(g_1^j(y), g_1^j(y_i)) <\epsilon, \ g_1^{mp_2}(y) = y, \ \text{for} \ p_2>M(\epsilon)+b_k-a_1,
\end{equation}
 for each $a_i\leq j\leq b_i$, $i = 1, 2, \ldots, k$ and each $m\in\mathbb{N}$. Let $p$ be the least common multiple of $p_1$ and $p_2$, then $p>M(\epsilon)+b_k-a_1$ and $h_1^{pm}(x, y) = (f_1^{pm}\times g_1^{pm})(x, y) = (f_1^{pm}(x), g_1^{pm}(y)) = (x, y)$, for each $m\in\mathbb{N}$. Using (\ref{p1}), (\ref{p2}) and the definition of the product metric $\tilde{d}$, we get $\tilde{d}(h_1^j(x, y), h_1^j(x_i, y_i))<\epsilon$, for each $a_i\leq j\leq b_i$, $i = 1, 2, \ldots, k$. Thus, $(X\times Y, f_{1,\infty}\times g_{1,\infty})$ has SSP.
 
 Conversely, suppose $ h_{1,\infty} = f_{1,\infty}\times g_{1,\infty}$ has SSP. Let $\delta>0$ be arbitrary and consider a sequence of non-negative integers $a_1\leq b_1< a_2 \leq b_2< \cdots <a_k \leq b_k$  with $a_j - b_{j-1}> M(\delta)$, where $M(\delta)$ is the positive integer as in the definition of SSP for $h_{1,\infty}$. By SSP of $h_{1,\infty}$, for any $(x_i, y_i)\in X\times Y$, $1\leq i\leq k$, there exists $(x, y)\in X\times Y$ such that $\tilde{d}(h_1^j(x, y), h_1^j(x_i, y_i))<\delta$, for $a_i\leq j\leq b_i$ and $h_1^{pm}(x, y) = (x, y)$, for $p>M(\delta)+b_k-a_1$ and each $m\in\mathbb{N}$. Now, $\tilde{d}(h_1^j(x, y), h_1^j(x_i, y_i)) = \max\{d_1(f_1^j(x), f_1^j(x_i)), d_2(g_1^j(y), g_1^j(y_i))\}<\delta$, which implies that $d_1(f_1^j(x), f_1^j(x_i))<\delta$ and $d_2(g_1^j(y), g_1^j(y_i)) <\delta$, $a_i\leq j\leq b_i$, $1\leq i\leq k$. Also, $f_1^{pm}\times g_1^{pm}(x, y) = (f_1^{pm}(x), g_1^{pm}(y)) = (x , y)$, implying that $f_1^{pm}(x) = x$ and $g_1^{pm}(y) = y$, for $p>M(\delta)+b_k-a_1$ and each $m\in\mathbb{N}$. Thus, both $f_{1,\infty}$ and $g_{1,\infty}$ have SSP.
 \end{proof}
 \begin{rmk}\label{RR} We have the following conclusions. \begin{enumerate} \item
 Above result is true for any finite product by induction.
 \item
 By similar arguments, Theorems \ref{K}$-$\ref{KK} and Corollary \ref{ck} are also true for WSP and QSP.\end{enumerate}
 \end{rmk}
\section{Specification Properties and Chaos}
In this section, we first show that for a non-autonomous system, QSP is equivalent to topological mixing on a compact metric space. Counter example is given to justify that result is not true when either the space is not compact or the family is not surjective. It is shown that a $k$-periodic non-autonomous system on intervals having WSP or QSP is Devaney chaotic and if the system has SSP, then this result is true in general.
\begin{thm}\label{Q} A non-autonomous system $(X, f_{1,\infty})$ has QSP if and only if $(X, f_{1,\infty})$ is topologically mixing.
\end{thm}
\begin{proof}
Suppose $(X, f_{1,\infty})$ has QSP.  Let $U$, $V\subseteq X$ be any two non-empty open subsets, then we can find $\epsilon>0$ such that $B_d(x, \epsilon)\subseteq U$ and $B_d(y, \epsilon)\subseteq V$, for any $x\in U$, $y\in V$. Let $M(\epsilon)$ be as in the definition of QSP. For any $n\geq M(\epsilon)$, using surjectivity of $(X, f_{1,\infty})$, for $y\in V$, there exists $z\in X$ such that $f_1^n(z) = y$. By QSP of $f_{1,\infty}$, we have  existence of a $w\in X$ such that $d(x, w)<\epsilon$ and $d(f_1^n(w), f_1^n(z))<\epsilon$. Therefore, we get $w\in B_d(x, \epsilon) \subseteq U$ and $f_1^n(w)\in B_d(y, \epsilon)\subseteq V$ and hence $f_1^n(U)\cap V\ne\emptyset$, for all $n\geq M(\epsilon)$. Thus, $f_{1,\infty}$ is topologically mixing.

Conversely, let $f_{1,\infty}$ be topologically mixing. Since $X$ is compact, therefore for any  open cover $\{x\in X: B_d(x, \epsilon/2)\}$ of $X$, there exists $\{x_i\}_{i=1}^k$ such that $X = \bigcup_{i=1}^k B_d(x_i, \epsilon/2)$, for any $\epsilon>0$. Now, by topological mixing of $f_{1,\infty}$, there exists a positive integer $M(\epsilon/2)$ such that \begin{equation}\label{Q1} f_1^n(B_d(x_i, \epsilon/2)) \cap B_d(x_j, \epsilon/2)\ne\emptyset, \ \text{for all} \ n\geq M(\epsilon/2), \ \text{and any} \ 1\leq i, j \leq k.
\end{equation}
Let $y_1$, $y_2\in X$ be arbitrary, then for any $n\geq M(\epsilon/2)$, $y_1$, $f_1^n(y_2)\in X$, there exist $1\leq r, s\leq k$ such that $y_1\in B_d(x_r, \epsilon/2)$ and $f_1^n(y_2)\in B_d(x_s, \epsilon/2)$. Also, by (\ref{Q1}), there exists $z\in B_d(x_r, \epsilon/2)\subseteq X$ such that $f_1^n(z)\in B_d(x_s, \epsilon/2)$. Thus, using triangle inequality we get that $d(z, y_1)<\epsilon$ and $d(f_1^n(z), f_1^n(y_2))<\epsilon$ implying that $(X, f_{1,\infty})$ has QSP.
\end{proof}
If we remove the condition of surjection of the family $f_{1,\infty}$ or the condition of compactness from the above theorem, then QSP may not be equivalent to topological mixing. The following example justifies this.
\begin{exm}Consider the non-autonomous system $(\mathbb{R}, f_{1,\infty})$, where $f_n(x) = 1/2^n$, for each $n\in\mathbb{N}$ and $f_{1,\infty} = \{f_n\}_{n=1}^\infty$. Then $f_{1,\infty}$ is not surjective. Since for any $z\in\mathbb{R}$, $\lim_{k\to\infty} f_1^k(z) = 0$, therefore for any $\epsilon>0$, there exists $M(\epsilon)>0$ such that $|f_1^k(z) - 0 |< \epsilon/2$, for all $k\geq M(\epsilon)$. For any non-negative integers $a_1\leq b_1<a_2 \leq b_2$ with $a_2-b_1> M(\epsilon)$ and any pair of points $x_1$, $x_2\in\mathbb{R}$, choosing $z = y_1$, we get that $|f_1^j(z) - f_1^j(y_1)| = 0<\epsilon$, for all $a_1\leq j\leq b_1$ and for all $a_2\leq j\leq b_2$, we have $|f_1^j(z) - f_1^j(y_2)|\leq |f_1^j(z)| + |f_1^j(y_2)|<\epsilon/2 + \epsilon/2 = \epsilon$. Thus, $f_{1,\infty}$ satisfies all the conditions of WSP for $s = 2$ and hence of QSP but $(\mathbb{R}, f_{1,\infty})$ is not topologically transitive and hence cannot be topologically mixing.
\end{exm}
Next, we provide an example showing that QSP need not imply WSP.
\begin{exm}Let $S^1$ be the unit circle and  $f_n : S^1 \to S^1$ be given by $f_n(e^{i\theta}) = e^{i\frac{n+1}{n}\theta}$, for each $n\in\mathbb{N}$. Let $f_{1,\infty} = \{f_n\}_{n=1}^\infty$, then $(S^1, f_{1,\infty})$ is topologically mixing on compact metric space $S^1$ with each $f_n$ being surjective and hence by Theorem \ref{Q}, the given system has QSP. In \cite[Example 2.2]{12}, authors have proved that $(S^1, f_{1,\infty})$ does not possess WSP. Thus, $(S^1, f_{1,\infty})$ has QSP but does not have WSP.
\end{exm}
By \cite[Lemma 13]{19} and Theorem \ref{Q}, we have the following result on QSP.
\begin{cor}\label{c4.1} Let $(X, f_{1,\infty})$ be a $k$-periodic non-autonomous system and $g = f_k\circ f_{k-1}\circ\cdots \circ f_1$. Then $(X, f_{1,\infty})$ has QSP if and only if  the corresponding autonomous system $(X, g)$  has QSP.
\end{cor}
\begin{rmk} We know that every topologically mixing non-autonomous  system has sensitive dependence on initial conditions, so if the non-autonomous system has SSP or WSP or the QSP, then it has sensitive dependence on initial conditions. 
\end{rmk}
\begin{cor}If the non-autonomous system $(X, f_{1,\infty})$ has WSP (QSP), then it is Wiggins chaotic.
\end{cor}
We know that for autonomous dynamical systems, WSP on intervals implies it is Devaney chaotic. In \cite{16}, authors have shown that on unit intervals topological transitivity need not imply Devaney chaos. Therefore, WSP on intervals may not imply Devaney chaos. We have  following result giving a condition under which WSP (QSP) implies Devaney chaos.
\begin{thm}\label{DC} Let $(I, f_{1,\infty})$ be a $k$-periodic non-autonomous system, where $I$ is any interval. If the system $(I, f_{1,\infty})$ has WSP (QSP), then it is Devaney chaotic. 
\end{thm}
\begin{proof}Since $(I, f_{1,\infty})$ has WSP (QSP), therefore by Remark \ref{RR}, the corresponding autonomous system $(I, f_k\circ\cdots\circ f_1)$ has WSP (QSP) and hence is topologically mixing. Now in autonomous systems, since topologically mixing on intervals implies Devaney chaos, therefore $(I, f_k\circ\cdots\circ f_1)$ has dense set of periodic points.  By \cite[Lemma 1]{19}, we get that $(I, f_{1,\infty})$ has dense set of periodic periodic points. Thus, $(I, f_{1,\infty})$ is  Devaney chaotic.\end{proof}
By Theorem \ref{Q} and Theorem \ref{DC}, we have the following result related to topological mixing for non-autonomous systems.
\begin{cor}Let $(I, f_{1,\infty})$ be a $k$-periodic non-autonomous system, where $I$ is any interval. If the system $(I, f_{1,\infty})$ is topologically mixing, then it is Devaney chaotic.
\end{cor}
Next, we show that if the non-autonomous system has SSP, then it directly implies the system is Devaney chaotic.
\begin{thm}If the non-autonomous system $(X, f_{1,\infty})$ has SSP, then it has dense set of periodic points.
\end{thm}
\begin{proof}Let $x\in X$ be arbitrary and $U$ be an open neighborhood of $x$.  We need to show that there is a periodic point in $X$ intersecting $U$. We have $x\in U$, so there is an $\epsilon>0$ such that $B(x, \epsilon)\subseteq U$. Let $M(\epsilon)$ be as in the definition of SSP for $f_{1,\infty}$ and any non-negative integers $a_1\leq b_1< a_2 \leq b_2< \cdots <a_k \leq b_k$  with $a_j - b_{j-1}> M(\epsilon)$ and any $x_1$, $x_2$, \ldots, $x_k\in X$. Therefore, by SSP of $(X, f_{1,\infty})$, there exists a periodic point $z\in X$, of period $p$, that is, $f_1^{pm}(z) = z$, for every $m\in\mathbb{N}$ such that $d(f_1^j(z), f_1^j(x_i))<\epsilon$, for all $a_i\leq j\leq b_i$, $1\leq i\leq k$. In particular for $x\in X$ and $a_1 = b_1 = 0$ in the above sequence, we get $d(f_1^0(z), f_1^0(x))<\epsilon$, that is, $d(z, x)<\epsilon$ which implies $z\in B(x, \epsilon)\subseteq U$. Hence, $(X, f_{1,\infty})$ has dense set of periodic points.
\end{proof}
The following example justifies that the converse of the above theorem is not true in general.
\begin{exm}Let $f$ be any bijective continuous self map on $X$. Consider the non-autonomous system $(X, f_{1,\infty})$, where $f_{1,\infty} = \{f, f^{-1}, f, f^{-1}, f, f^{-1}, \ldots\}$. Since $f_1^{2k}(x) = x$, for each $k\in\mathbb{N}$ and any $x\in X$, therefore every point of $X$ is periodic and hence $(X, f_{1,\infty})$ has dense set of periodic points. But as $\mathcal{O}_{f_{1,\infty}}(x) = \{x, f(x)\}$, so $(X, f_{1,\infty})$ can never be topologically transitive and thus cannot possess SSP.
\end{exm}
\begin{cor}If the non-autonomous system $(X, f_{1,\infty})$ has SSP, then it Devaney chaotic.
\end{cor}
Converse of the above result is not true in general as shown in the following example.
\begin{exm}\label{2} Let $\Sigma_2 =\{0, 1\}^\mathbb{Z}$ = $\{( \ldots, x_{-2}, x_{-1}$, \fbox{$x_0$}, $x_1, \ldots): x_i \in \{0,1\}, \ \text{for every}$ \  $i  \in \mathbb{Z}\}$ with metric  \[\rho(x,y) = \sum_{j=-\infty}^{\infty}\frac{|x_j-y_j|}{2^{|j|}} \]  
for any pair $ x= ( \ldots, x_{-2}, x_{-1}$, \fbox{$x_0$}, $ x_1, \ldots)$;  $y = ( \ldots, y_{-2}, y_{-1}$, \fbox{$y_0$}, $ y_1, \ldots) \in \Sigma_2$.  Define $\sigma: \Sigma_2 \to \Sigma_2$ by $\sigma(x) = ( \ldots, x_{-2}, x_{-1}, x_0, \fbox{$x_1$}, x_2, \ldots)$, where $x = ( \ldots, x_{-2}, x_{-1}, $ \fbox{$x_0$}, $ x_1, x_2 \ldots) \in \Sigma_2$, then $\sigma$ is a homeomorphism and is called the \emph{shift map} on $\Sigma_2$. Consider the non-autonomous system $(\Sigma_2, f_{1,\infty})$, where \[f_{1,\infty} = \{\sigma, \sigma^{-1}, \sigma^2, \sigma^{-2}, \sigma^3, \sigma^{-3}, \ldots\}.\] Let $U$ and $V$ be any two non-empty open subsets of $\Sigma_2$. Since $\sigma$ is topologically mixing, therefore there exists $k\in\mathbb{N}$ such that $\sigma^n(U) \cap V\ne\emptyset$, for all $n\geq k$. Now, $f_1^{2k-1} = \sigma^k$, which implies that $f_1^{2k-1}(U) \cap V\ne\emptyset$ and hence $f_{1,\infty}$ is topologically transitive. Note that $N_{f_{1,\infty}}(U, V) = \{2k-1, 2k+1, 2k+3, \ldots\}$, so $f_{1,\infty}$ cannot be topologically mixing. Similarly, using sensitivity of $\sigma$, it can be shown that $f_{1,\infty}$ is sensitive. Also, $f_1^{2m}(x) = x$, for each $m\in\mathbb{N}$ and any $x\in \Sigma_2$, implying that every point of $\Sigma_2$ is periodic.  Thus, $(\Sigma_2, f_{1,\infty})$ is Devaney chaotic but it is not topologically mixing and hence cannot have  any kind  of specification properties.\end{exm}
\section{Specification Properties of Induced Systems}
In this section, we study the specification properties of systems induced on hyperspaces and probability measures spaces. It is also proved that if the non-autonomous system has SSP, then the corresponding system induced on the hyperspaces also has SSP and this result holds both ways for WSP (QSP). It is proved that if the non-autonomous system has SSP, then the corresponding system induced on the probability measurable spaces also has SSP and this result holds both ways for QSP.
\begin{thm}\label{T1}If the non-autonomous system $(X, f_{1,\infty})$ has SSP, then ($\mathcal{K}(X), \overline{f}_{1,\infty})$ also has SSP. 
\end{thm}
\begin{proof}
Let $\epsilon>0$ be arbitrary and $M(\epsilon/2)$ be the positive integer as in the definition of SSP. Let $A_1$, $A_2$, \ldots, $A_k\in\mathcal{K}(X)$ and $a_1\leq b_1< a_2 \leq b_2< \cdots <a_k \leq b_k$, be a sequence of any non-negative integers with $a_j - b_{j-1}> M(\epsilon/2)$, $2\leq j\leq k$ and any $p> M(\epsilon/2) + b_k - a_1$. We have $\overline{f}_1^j$ is continuous on a compact metric space, for every $j\geq 0$, so for every $\epsilon>0$, there exists $\delta>0$ such that 
\begin{equation}\label{H1}
\mathcal{H}(A, B) <\delta \ \implies \ \mathcal{H}(\overline{f}_1^j(A), \overline{f}_1^j(B))<\epsilon/2, \ \text{for every} \ A, B\in\mathcal{K}(X).
\end{equation} 
Now, $\mathcal{F}(X)$ is dense in $\mathcal{K}(X)$, therefore there exist $B_1$, $B_2$, \ldots, $B_k\in \mathcal{F}(X)$ such that $\mathcal{H}(A_i, B_i)< \delta$ and hence by (\ref{H1}), we get that 
\begin{equation}\label{H2}
\mathcal{H}(\overline{f}_1^j(A_i), \overline{f}_1^j(B_i))<\epsilon/2, \ \text{for} \ i = 1, 2, \ldots, k.
\end{equation}
Let $B_l = \{x_i^l\}_{i=1}^n$, $l = 1, 2, \ldots, k$. By SSP of $(X, f_{1,\infty})$, there exist $z_i\in X$ for each $x_i^l$ such that $d(f_1^j(z_i), f_1^j(x_i^l))< \epsilon/2$, for each $a_l\leq j\leq b_l$, $l = 1, 2, \ldots, k$ and $f_1^{m p_i}(z_i) = z_i$, for each $m\in\mathbb{N}$ and $i = 1, 2, \ldots, n$. Taking $p$ to be the least common multiple of $p_1$, $p_2$, \ldots, $p_n$, we get $f_1^{mp}(z_i) = z_i$. Let $C = \{z_i\}_{i=1}^n$, then
\begin{align}\label{H3} \mathcal{H}(\overline{f}_1^j(C), \overline{f}_1^j(B_l)) = d(f_1^j(z_i), f_1^j(x_i^l))< \epsilon/2, \ a_l\leq j\leq b_l, \ l=1, 2, \ldots, k.
\end{align} Thus, using (\ref{H2}), (\ref{H3}) and triangle inequality, we get that $\mathcal{H}(\overline{f}_1^j(C), \overline{f}_1^j(A_l))<\epsilon$, for each $a_l\leq j\leq b_l, \ l=1, 2, \ldots, k$  and $\overline{f}_1^{mp}(C) = C$, for each $m\in\mathbb{N}$. Therefore, $(\mathcal{K}(X), \overline{f}_{1,\infty})$ has SSP.
\end{proof}
For WSP, we have the following result.
\begin{thm}\label{5.1}A non-autonomous system $(X, f_{1,\infty})$ has WSP if and only if $(\mathcal{K}(X), \overline{f}_{1,\infty})$ has WSP.
\end{thm}
\begin{proof}
Suppose that $(\mathcal{K}(X), \overline{f}_{1,\infty})$ has WSP. Let $\epsilon>0$ be arbitrary and $M(\epsilon)$ be the positive integer as in the definition of WSP. Let $x_1$, $x_2$, \ldots, $x_k\in X$ and $a_1\leq b_1< a_2 \leq b_2< \cdots <a_k \leq b_k$, be a sequence of any non-negative integers with $a_j - b_{j-1}> M(\epsilon)$, $2\leq j\leq k$. Since $(\mathcal{K}(X), \overline{f}_{1,\infty})$ has WSP, therefore for any $A_1$, $A_2$, \ldots, $A_k\in\mathcal{K}(X)$, there exists $B\in\mathcal{K}(X)$ such that \begin{align}\label{HC}
\mathcal{H}(\overline{f}_1^j(B), \overline{f}_1^j(A_l)) < \epsilon, \ a_l\leq j\leq b_l, \ l=1, 2, \ldots, k.
\end{align}
Taking $A_i =\{x_i\}$, for each $i= 1, 2, \ldots, k$, we get $\mathcal{H}(\overline{f}_1^j(B), \overline{f}_1^j(\{x_l\})) < \epsilon, \ a_l\leq j\leq b_l, \ l=1, 2, \ldots, k$, where $\overline{f}_1^j(\{x_l\}) = \{f_1^j(x_l)\}$.

Assume if possible that $d(f_1^j(b), f_1^j(x_l))\geq\epsilon$, for each $b\in B$ and each $a_l\leq j\leq b_l, \ l=1, 2, \ldots, k$. Therefore, we get $f_1^j(B)\not\subset B_{\mathcal{H}}(\overline{f}_1^j(x_l), \epsilon)$, for each $j$, that is, $f_1^j(B)\not\subset N(\{f_1^j(x_l)\}, \epsilon)$ and $f_1^j(x_l)\notin N(f_1^j(B), \epsilon)$, which implies that $\mathcal{H}(\overline{f}_1^j(B), \overline{f}_1^j(\{x_l\})) \geq \epsilon, \ a_l\leq j\leq b_l, \ l=1, 2, \ldots, k, \ i= 1, 2, \ldots, k$, which is a contradiction to (\ref{HC}). Thus, there exist $b\in B\subseteq X$ such that $d(f_1^j(b), f_1^j(x_l))<\epsilon$, for all $a_l\leq j\leq b_l, \ l=1, 2, \ldots, k$ implying that $(X, f_{1,\infty})$ has WSP. Converse is a direct consequence of Theorem \ref{T1}.
\end{proof}
\begin{cor}\label{5.2} The non-autonomous system $(\mathcal{K}(X\times Y), \overline{f}_{1,\infty}\times \overline{g}_{1,\infty})$ has WSP if and only if $(\mathcal{K}(X), \overline{f}_{1,\infty})$ and $(\mathcal{K}(Y), \overline{g}_{1,\infty})$ has WSP.
\end{cor}
\begin{rmk}\label{r5.1} By similar arguments, Theorem \ref{5.1} and Corollary \ref{5.2} are also true for the non-autonomous systems having QSP.
\end{rmk} 
\begin{thm}\label{PP}If the non-autonomous system $(X, f_{1,\infty})$ has SSP, then $(\mathcal{M}(X), \widetilde{f}_{1,\infty})$ also has SSP.
\end{thm}
\begin{proof}
Let $\epsilon>0$ be arbitrary and $M(\epsilon/2)$ be the positive integer as in the definition of SSP. Let $\mu_1$, $\mu_2$, \ldots, $\mu_k\in \mathcal{M}(X)$ be given and $a_1\leq b_1< a_2 \leq b_2< \cdots <a_k \leq b_k$, be a sequence of any non-negative integers with $a_j - b_{j-1}> M(\epsilon/2)$, $2\leq j\leq k$ and any $p> M(\epsilon/2) + b_k - a_1$. Now, each $\widetilde{f}_{i}$ is  continuous from $\mathcal{M}(X)$ to itself and $\mathcal{M}(X)$ is compact, therefore for given $\epsilon>0$, there exists $\eta>0$ such that 
\begin{equation}\label{M1}
\mathcal{D}(\mu, \nu) <\eta \ \implies \ \mathcal{D}(\widetilde{f}_{1}^j(\mu), \widetilde{f}_{1}^j(\nu))<\epsilon/2, \ \text{for every}, \ \mu, \nu\in \mathcal{M}(X) \ \text{and} \ a_l\leq j\leq b_l,
\end{equation} where $1\leq l \leq k$. 
Now, $\mu_i\in \mathcal{M}(X)$ and $\mathcal{M}_\infty(X)$ is dense in $\mathcal{M}(X)$, therefore there exist $\nu_1$, $\nu_2$, \ldots, $\nu_k\in \mathcal{M}_n(X)$ such that $\mathcal{D}(\mu_i, \nu_i)< \eta$ and hence by (\ref{M1}), we get
\begin{equation}\label{M2}
\mathcal{D}(\widetilde{f}_{1}^j(\mu_i), \widetilde{f}_{1}^j(\nu_i))<\epsilon/2, \ \text{for} \ i = 1, 2, \ldots, k.
\end{equation}
Let $\nu_i = (\sum_{l=1}^n\delta_{x_l^i})/n$, for $i =1, 2, \ldots, k$. Since $f_{1,\infty}$ satisfies SSP, therefore there exist $z_l\in X$ such that $f_1^{pm}(z_l) = z_l$, for every $m\in\mathbb{N}$ and $d(f_1^j(z_l), f_1^j(x_l^i)) < \epsilon/2$, for each $a_i\leq j\leq b_i$, $i = 1, 2, \ldots, k$ and $l= 1, 2, \ldots, n$. Let $\rho = (\sum_{l=1}^n\delta_{z_l})/n$ and $A$ be any Borel measurable set, then 
\begin{align*}
\widetilde{f}_{1}^{pm}(\rho)(A)  =  \rho(f_1^{-pm}(A)) & = \frac{1}{n}(\delta_{z_1} + \cdots + \delta_{z_n})(f_1^{-pm}(A)) \\ & = \frac{1}{n}(\delta_{z_1}(f_1^{-pm}(A)) + \cdots + \delta_{z_n}(f_1^{-pm}(A)))\\ & = \frac{1}{n} \sum_{l=1}^n\delta_{z_l}(A),
\end{align*}
because $\delta_{z_l}(f_1^{-pm}(A)) = \delta_{z_l}(A)$ using $f_1^{pm}(z_l) = z_l$ and hence $\widetilde{f}_{1}^{pm}(\rho)(A) = \rho(A)$, for any $A\in\mathcal{B}(X)$. Thus,  $\widetilde{f}_{1}^{pm}(\rho) = \rho$, for every $m\in\mathbb{N}$ and we have
\begin{align}\label{M3}\mathcal{D}(\widetilde{f}_{1}^j(\rho), \widetilde{f}_{1}^j(\nu_i))< \epsilon/2, \ a_i\leq j\leq b_i, \ i=1, 2, \ldots, k.
\end{align} Therefore, using (\ref{M2}), (\ref{M3}) and triangle inequality, we get that $\mathcal{D}(\widetilde{f}_{1}^j(\rho), \widetilde{f}_{1}^j(\mu_i))<\epsilon$, for each $a_i\leq j\leq b_i, \ i=1, 2, \ldots, k$ and hence $(\mathcal{M}(X), \widetilde{f}_{1,\infty})$ has SSP.
\end{proof}
\begin{rmk}By similar arguments, Theorem \ref{PP} also holds for weak specification property.
\end{rmk}
Next, we show that for the systems having QSP above result hold both ways.
\begin{thm} A non-autonomous system $(X, f_{1,\infty})$ is topological mixing if and only if $(\mathcal{M}(X), \widetilde{f}_{1,\infty})$ is topological mixing. 
\end{thm}
\begin{proof}
First suppose that $(X, f_{1,\infty})$ is topological mixing. Let $W_1$ and $W_2$ be any two non-empty open subsets of $\mathcal{M}(X)$. Since $\mathcal{M}_\infty(X)$ is dense in $\mathcal{M}(X)$, therefore there exists $\mu_j\in\mathcal{M}_\infty(X)$ such that $\mu_j = (\sum_{i=1}^m\delta_{x_i^j})/m\in W_j$, for each $j\in\{1, 2\}$. We can choose open neighborhood $U_i^j$ of $x_i^j$ in such a manner if $y_i^j\in U_i^j$, then $(\sum_{i=1}^m\delta_{x_i^j})/m\in W_j$, for each $j\in\{1, 2\}$. It is easy to see that if $f_{1,\infty}$ is topological mixing, then $\underbrace{f_{1,\infty} \times f_{1,\infty} \times \cdots \times f_{1,\infty}}_\text{$m$-times}$ is  also topological mixing. Therefore, there exists $k\in\mathbb{N}$ such that $f_1^n(U_i^1)\cap U_i^2\ne\emptyset$, for all $n\geq k$ and for every $i\in \{1, 2, \ldots, m\}$. Let  $z_i^2\in f_1^n(U_i^1)$ and $z_i^2\in U_i^2$, that is, $f_1^{-n}(z_i^2)\in U_i^1$ and $z_i^2\in U_i^2$, for all $n\geq k$ and for every $i\in\{1, 2, \ldots, m\}$ and hence  we get that $\nu = (\sum_{i=1}^m\delta_{z_i^2})/m\in W_2$ and $\widetilde{f}_1^{-n}(\nu) = (\sum_{i=1}^r\delta_{f_1^{-n}(z_i^2)})/m\in W_1$. Thus, $\widetilde{f}_1^n(W_1)\cap W_2\ne\emptyset$, for all $n\geq k$ and for every $j\in\{1, 2, \ldots, k\}$, giving that $(\mathcal{M}(X), \widetilde{f}_{1,\infty})$ is topological mixing. 

Conversely, let $U$, $V$ be non-empty open subsets of $X$. Let $W_1 = \{\mu\in\mathcal{M}(X) : \mu(U)> 4/5\}$ and $W_2 = \{\mu\in \mathcal{M}(X) : \mu(V)>4/5\}$, then $W_1$ and $W_2$ are non-empty open subsets of $\mathcal{M}(X)$. If $\mu_n\to \mu$ such that $\mu_n\in\mathcal{M}(X)\setminus W_1$, then $\mu_n(U)\leq 4/5$ implying that $\mu(U)\leq 4/5$ and hence $\mu\in\mathcal{M}(X)\setminus W_1$. Thus, $W_1$ is open in $\mathcal{M}(X)$ and similarly $W_2$ is also open in $\mathcal{M}(X)$. Now, since the system $(\mathcal{M}(X), \widetilde{f}_{1,\infty})$ is topological mixing, therefore there exists $k\in\mathbb{N}$ such that $\widetilde{f}_1^n(W_1)\cap W_2\ne\emptyset$, for all $n\geq k$. This implies that there exists $\nu\in W_1$ with $\widetilde{f}_1^n(\nu)\in W_2$ and hence $\widetilde{f}_1^n(\nu)(V) = \nu(f_1^{-n}(V))>4/5$, for all $n\geq k$. Also, as $\nu(U)>4/5$, so $f_1^n(U)\cap V\ne\emptyset$, for all $n\geq k$. Thus, $(X, f_{1,\infty})$ is topological mixing.
\end{proof}
By above theorem and Theorem \ref{Q}, we have the following result.
\begin{cor}\label{c5.2} The non-autonomous system $(X, f_{1,\infty})$ has QSP if and only if $(\mathcal{M}(X)$, $\widetilde{f}_{1,\infty})$ has QSP.
\end{cor}
\begin{cor}\label{c5.3} The non-autonomous system $(\mathcal{M}(X\times Y), \widetilde{f}_{1,\infty}\times \widetilde{g}_{1,\infty})$ has QSP if and only if $(\mathcal{M}(X), \widetilde{f}_{1,\infty})$ and $(\mathcal{M}(Y), \widetilde{g}_{1,\infty})$ has QSP.
\end{cor}
Now, we give an example supporting results of this section.
\begin{exm} Consider a $3$-periodic non-autonomous system $(X, f_{1,\infty})$, where $f_1 = \sigma$, $f_2 = \sigma^{-2}$, $f_3 = \sigma^2$ and $X = \Sigma_2$, that is, 
$f_{1,\infty} = \{\sigma, \sigma^{-2}, \sigma^2, \sigma, \sigma^{-2}, \sigma^2, \ldots\}$, where $\sigma$ is the shift map as defined in Example \ref{2}. Since, $f_3\circ f_2\circ f_1 = \sigma$ and $\sigma$ has QSP, therefore the induced autonomous system $(X, f_3\circ f_2\circ f_1)$ has QSP and hence by Corollary \ref{c4.1}, $(\Sigma_2, f_{1,\infty})$ has QSP. Thus, by Remark \ref{r5.1} and Corollary \ref{c5.2}, both the systems  $(\mathcal{K}(\Sigma_2), \overline{f}_{1,\infty})$ and $(\mathcal{M}(\Sigma_2), \widetilde{f}_{1,\infty})$ have QSP. Also, if $g_1: I \to I$ is a continuous map on a closed unit interval $I$, given by $g_1 = 4x(1-x)$ and $g_2: I\to I$ is the identity map and  $(I, g_{1,\infty})$ is the corresponding  $2$-periodic non-autonomous system, then $(I, g_{1,\infty})$ has QSP. Therefore, by Remark \ref{RR}, we have $(\Sigma_2\times I, f_{1,\infty}\times g_{1,\infty})$ has QSP and respectively by Remark \ref{r5.1} and Corollary \ref{c5.3}, the systems $(\mathcal{K}(\Sigma_2\times I), \overline{f}_{1,\infty}\times \overline{g}_{1,\infty})$ and $(\mathcal{M}(\Sigma_2\times I), \widetilde{f}_{1,\infty}\times \widetilde{g}_{1,\infty})$ have QSP.
\end{exm}


\begin{thebibliography}{80}
\bibitem{1} Bartoll, Salud, Félix Martínez-Giménez, and Alfredo Peris, \emph{The specification property for backward shifts}, Journal of Difference Equations and Applications 18, no. 4 (2012): 599-605.
\bibitem{2} Bauer, Walter, and Karl Sigmund, \emph{Topological dynamics of transformations induced on the space of probability measures}, Monatshefte für Mathematik 79, no. 2 (1975): 81-92.
\bibitem{3} Bowen, Rufus, \emph{Periodic points and measures for Axiom A diffeomorphisms}, Transactions of the American Mathematical Society 154 (1971): 377-397.
\bibitem{4} Coutinhoa, F. A. B., M. N. Burattinia, L. F. Lopeza, and E. Massada, \emph{Threshold conditions for a non-autonomous epidemic system describing the population dynamics of dengue}, Bulletin of mathematical biology 68, no. 8 (2006): 2263-2282.
\bibitem{5} Dateyama, Masahito, \emph{The almost weak specification property for ergodic group automorphisms of abelian groups}, Journal of the Mathematical Society of Japan 42, no. 2 (1990): 341-351.
\bibitem{7} Huber, Peter J., and Elvezio M. Ronchetti, \emph{Robust statistics},   Wiley Series in Probability and Statistics, John Wiley $\&$ Sons, Inc., Hoboken, NJ, second ed., 2009.
\bibitem{8} Illanes, Alejandro, and Sam Nadler, \emph{Hyperspaces}, fundamentals and recent advances. Vol. 216. CRC Press, 1999.
\bibitem{9} Kolyada, Sergiĭ, and Lubomır Snoha, \emph{Topological entropy of nonautonomous dynamical systems}, Random and computational dynamics 4, no. 2 (1996): 205-233.
\bibitem{10} Kulczycki, Marcin, Dominik Kwietniak, and Piotr Oprocha, \emph{On almost specification and average shadowing properties}, Fundamenta Mathematicae 224, no. 3 (2014): 241-278.
\bibitem{11} Li, Tien-Yien, and James A. Yorke, \emph{Period three implies chaos}, The American Mathematical Monthly 82, no. 10 (1975): 985-992.
\bibitem{101} Mazur, Marcin, and Piotr Oprocha, \emph{Subshifts, rotations and the specification property}, Topological Methods in Nonlinear Analysis 46, no. 2 (2015): 799-812.
\bibitem{12} Memarbashi, Reza, and Hosein Rasuli, \emph{Notes on the dynamics of nonautonomous discrete dynamical systems}, Journal of Advanced Research in Dynamical and Control Systems 6, no. 2 (2014): 8-17.
\bibitem{13} Miralles, Alejandro, Marina Murillo-Arcila, and Manuel Sanchis, \emph{Sensitive dependence for nonautonomous discrete dynamical systems}, Journal of Mathematical Analysis and Applications 463, no. 1 (2018): 268-275.
\bibitem{14} Pfister, Charles-Edouard, and Wayne G. Sullivan, \emph{Large deviations estimates for dynamical systems without the specification property. Application to the $\beta$-shifts}, Nonlinearity 18, no. 1 (2004): 237-261.
\bibitem{15} Salman, Mohammad, and Ruchi Das, \emph{Multi-sensitivity and other stronger forms of sensitivity in non-autonomous discrete systems}, Chaos Solitons Fractals, 115 (2018): 341-348.
\bibitem{16} Sánchez, Iván, Manuel Sanchis, and Hugo Villanueva, \emph{Chaos in hyperspaces of nonautonomous discrete systems}, Chaos Solitons Fractals 94 (2017): 68-74.
\bibitem{17} Sarkooh, J. Nazarian, and F. H. Ghane, \emph{Specification and thermodynamic properties of non-autonomous dynamical systems}, arXiv preprint arXiv:1712.06109 (2017).
\bibitem{18} Shah, Sejal,  Ruchi Das, and  Tarun Das, \emph{Specification property for topological spaces}, Journal of Dynamical and Control Systems 22, no. 4 (2016): 615-622.
\bibitem{19} Sharma, Puneet, and Manish Raghav \emph{Dynamics of Non-Autonomous Discrete Dynamical Systems}, Topology Proceedings, vol. 52 (2018): 45-59.
\bibitem{21} Sigmund, Karl, \emph{On dynamical systems with the specification property},  Transactions of the American Mathematical Society 190 (1974): 285-299.
\bibitem{22} Thompson, Daniel, \emph{The irregular set for maps with the specification property has full topological pressure}, Dynamical Systems 25, no. 1 (2010): 25-51.
\bibitem{23} Tian, Chuanjun, and Guanrong Chen, \emph{Chaos of a sequence of maps in a metric space}, Chaos Solitons  Fractals, 28, no. 4 (2006): 1067-1075.
\bibitem{24} Vasisht, Radhika, and  Ruchi Das, \emph{On Stronger Forms of Sensitivity in Non-autonomous Systems}, Taiwanese Journal of Mathematics 22, no. 5 (2018): 1139-1159.
\bibitem{25} Zhang, Yan, Shujing Gao, and Yujiang Liu, \emph{Analysis of a nonautonomous model for migratory birds with saturation incidence rate}, Communications in Nonlinear Science and Numerical Simulation 17, no. 4 (2012): 1659-1672. 
\end{thebibliography}
\end{document}